\documentclass[12pt]{article}

\usepackage[margin=1in]{geometry}
\usepackage{amsmath,amssymb,amsthm,mathrsfs}
\usepackage{mathtools}
\usepackage{booktabs}
\usepackage{hyperref,float}
\usepackage{tikz}
\usetikzlibrary{arrows.meta,positioning}
\usepackage{enumitem}
\everymath{\displaystyle}
\hypersetup{
    colorlinks=true,
    linkcolor=blue,
    citecolor=blue,
    urlcolor=blue
}

\newcommand{\DL}{\mathcal{D}^{L}}
\newcommand{\comp}[1]{\overline{#1}}
\newcommand{\ceil}[1]{\left\lceil #1 \right\rceil}
\newcommand{\floor}[1]{\left\lfloor #1 \right\rfloor}
\newcommand{\Tr}{\operatorname{Tr}}
\newcommand{\diam}{\operatorname{diam}}

\theoremstyle{plain}
\newtheorem{theorem}{Theorem}[section]
\newtheorem{lemma}[theorem]{Lemma}

\newtheorem{corollary}[theorem]{Corollary}

\theoremstyle{definition}

\newtheorem{remark}[theorem]{Remark}

\renewenvironment{proof}{{\noindent\bfseries Proof. }}{\hfill$\square$\par}

\title{Extremal chromatic bounds for distance Laplacian eigenvalues}
\author{
    {\small Bilal Ahmad Rather}\\[2mm]
    {\small School of Mathematics and Statistics, Shandong University of Technology, Zibo 255049, China}\\
    \texttt{bilahamadrr@gmail.com}
}
\date{}
\begin{document}
\maketitle

\begin{abstract}
For a connected simple graph $G$ on $n$ vertices with chromatic number $\chi$, the distance Laplacian matrix is
$\DL(G)=\operatorname{diag}(\Tr_G(v_1),\dots,\Tr_G(v_n))-D(G)$, where $D(G)$ is the distance matrix and
$\Tr_G(v)=\sum_{u\in V(G)} d_G(u,v)$ is the transmission. The eigenvalues of $\DL(G)$ are ordered as
$\partial^{L}_1(G)\ge \partial^{L}_2(G)\ge \cdots \ge \partial^{L}_n(G)=0$.
Building on the chromatic lower bound $\partial^{L}_1(G)\ge n+\ceil{n/\chi}$ and subsequent developments, we prove a
\emph{color-class majorization principle}: if $(\ell_1,\dots,\ell_\chi)$ are the color-class sizes in an optimal
$\chi$-coloring with $\ell_1\ge\cdots\ge\ell_\chi$, then the first $\ell_1-1$ distance Laplacian eigenvalues satisfy
$\partial^{L}_i(G)\ge n+\ell_1$, for $1\le i\le \ell_1-1$. This gives sharp lower bounds on the number of eigenvalues
above the chromatic threshold $b_\chi=n+\ceil{n/\chi}$, thereby refining distribution theorems of [Aouchiche and Hansen, Filomat, 2017] and [Pirzada and Khan LAA, 2021]. We further refine clique/independent-set based multiplicity results by
deriving explicit chromatic criteria in terms of neighborhood compression, and we generalize the extremal problem for minimum
$\partial^{L}_1$ at fixed chromatic number by characterizing the balanced complete multipartite minimizers. Finally, we present a  Ky Fan type result, and complement-component consequences of the majorization principle. 
\end{abstract}



\section{Introduction}

A \emph{graph} $G=(V(G),E(G))$ is a finite, simple, undirected, connected structure, where
$V(G)=\{v_1,\dots,v_n\}$ is the vertex set and
$E(G)\subseteq \{\{v_i,v_j\}:1\le i<j\le n\}$ is the edge set. The \emph{order} and \emph{size} of $G$ are
$n=|V(G)|$ and $m=|E(G)|$, respectively. For $u,v\in V(G)$, the \emph{distance} $d_G(u,v)$ is the length of a shortest
$u$--$v$ path, and
\[
    \diam(G)=\max\{d_G(u,v):u,v\in V(G)\}.
\]
For a vertex $v\in V(G)$, its open neighborhood is
\[
    N_G(v)=\{u\in V(G):\{u,v\}\in E(G)\}.
\]
A \emph{proper vertex coloring} of $G$ is a partition of $V(G)$ into independent sets, called color classes. The least number
of such classes is the \emph{chromatic number} $\chi(G)$, see \cite{BrouwerHaemers2012,CvetkovicDoobSachs1995}. A cross-edge means an edge whose endpoints lie in two different color classes of a fixed proper coloring.

Distance-based matrices form an important interface between combinatorial structure and linear algebra. The adjacency matrix
$A(G)$ and the Laplacian matrix $L(G)=\Delta(G)-A(G)$, where $\Delta(G)$ is the diagonal degree matrix, encode local adjacency
information and lie at the centre of spectral graph theory \cite{BrouwerHaemers2012,CvetkovicDoobSachs1995}. Among their
most useful properties are that $L(G)$ is symmetric positive semidefinite, $L(G)\mathbf{1}=\mathbf{0}$, and the multiplicity
of the eigenvalue $0$ is the number of connected components of $G$. Hence, for a connected graph, this multiplicity is $1$.
These features make Laplacian spectra natural tools for measuring global connectivity and expansion, and they are often the
right setting for interlacing and extremal arguments \cite{BrouwerHaemers2012,CvetkovicDoobSachs1995}.

While $A(G)$ and $L(G)$ are based on edges, many applications depend primarily on distances: communication cost, chemical
structure descriptors, transportation, routing, and network latency all have this feature. This motivates the \emph{distance
matrix} $D(G)=(d_G(v_i,v_j))_{i,j=1}^n$ and its Laplacian-type companions. For a vertex $v$, the \emph{transmission} is
\[
    \Tr_G(v)=\sum_{u\in V(G)} d_G(u,v),
\]
and the diagonal transmission matrix is
$\operatorname{diag}(\Tr_G(v_1),\dots,\Tr_G(v_n))$. The \emph{distance Laplacian matrix} is
\[
    \DL(G)=\operatorname{diag}(\Tr_G(v_1),\dots,\Tr_G(v_n))-D(G).
\]
This matrix was introduced and systematically studied by Aouchiche and Hansen \cite{AouchicheHansen2013}. Much like the
classical Laplacian, $\DL(G)$ is symmetric positive semidefinite and satisfies $\DL(G)\mathbf{1}=\mathbf{0}$. Hence its
eigenvalues may be ordered as
\[
    \partial^{L}_1(G)\ge \partial^{L}_2(G)\ge \cdots \ge \partial^{L}_n(G)=0.
\]
Distance Laplacian spectra have been investigated in connection with transmission regularity, forbidden substructures,
multiplicities, and extremal problems
\cite{AouchicheHansen2013,AouchicheHansen2014,LinWuChenShu2016,NathPaul2014,FernandesEtAl2018,NiuFanWang2015,bilalMV}.
Recent surveys can be found in \cite{BilalsurveyDL,BilalsurveyDSL}.

A particularly fruitful direction concerns the interaction between distance Laplacian eigenvalues and graph invariants that
constrain global structure, such as the chromatic number. The chromatic number controls the extent to which a graph can be
decomposed into independent sets, and hence regulates the multipartite structure that a graph can admit. On the spectral side,
the largest distance Laplacian eigenvalue $\partial^{L}_1(G)$ plays the role of a distance-based spectral radius and is sensitive
to long-range connectivity patterns through $D(G)$ and the transmissions. Relating $\partial^{L}_1(G)$ and the remaining
$\partial^{L}_i(G)$ to $\chi(G)$ therefore provides a useful bridge between coloring structure and distance-driven spectral
quantities.

In this direction, Aouchiche and Hansen \cite{AouchicheHansen2017} established the sharp chromatic lower bound, for graphs
other than the complete graph,
\begin{equation}\label{eq:AH-bound}
    \partial^{L}_1(G) \ge n+\ceil{\frac{n}{\chi(G)}}\qquad (G\ncong K_n),
\end{equation}
and initiated a systematic study of how many distance Laplacian eigenvalues lie below or above the corresponding chromatic
threshold
\[
    b_\chi=n+\ceil{\frac{n}{\chi(G)}}.
\]
Their work \cite{AouchicheHansen2017} developed distribution phenomena around $b_\chi$ and exhibited structural conditions,
for example diameter constraints, that force additional eigenvalues to cross this threshold. Later, Pirzada and Khan
\cite{PirzadaKhan2021} strengthened these ideas by proving that, when $\chi(G)\le n-2$, the second largest eigenvalue also
obeys the chromatic bound $\partial^{L}_2(G)\ge b_\chi$, and by deriving further distribution results together with an extremal
characterization for graphs with large chromatic number. These results complement earlier investigations on distance Laplacian
spectra and extremal behavior
\cite{AouchicheHansen2013,AouchicheHansen2014,LinWuChenShu2016,NathPaul2014,FernandesEtAl2018,NiuFanWang2015,bilalacta,bilalijpam}
and follow the broader principle that spectral constraints can quantify global combinatorial restrictions
\cite{BrouwerHaemers2012,CvetkovicDoobSachs1995}.

The motivation of the present study is twofold. First, inequalities such as \eqref{eq:AH-bound} naturally lead to more precise
questions: how many eigenvalues must exceed $b_\chi$; how do these counts depend on diameter, complement connectivity, and twin
classes; and which graphs minimize $\partial^{L}_1(G)$ under chromatic constraints? Second, from an applied viewpoint, the
spectrum of $\DL(G)$ encodes distance and transmission information. Lower bounds and distribution results in terms of
$\chi(G)$ therefore provide robust and easily computed certificates for network complexity, routing difficulty, and heterogeneity,
especially when $\chi(G)$ or good bounds on it are available.

Accordingly, we develop refined lower bounds and distribution theorems that strengthen and generalize the main results of
\cite{AouchicheHansen2017,PirzadaKhan2021}, while also connecting to multiplicity mechanisms and extremal principles studied in
\cite{AouchicheHansen2013,AouchicheHansen2014,LinWuChenShu2016,NathPaul2014,FernandesEtAl2018,NiuFanWang2015}. Our approach
combines edge-monotonicity with explicit spectra of complete multipartite graphs and neighborhood-compression methods. The main
contribution is an eigenvalue multiplicity mechanism that comes directly from an optimal coloring: a large color class forces not
only one large distance Laplacian eigenvalue, but a whole block of them.

The paper is organized as follows. Section \ref{section 2} records basic and existing results. Section \ref{section 3} proves the
color-class majorization principle and its first consequences. Section \ref{section 4} gives structural results in terms of twin
cliques, twin independent sets, and distance Laplacian eigenvalues. Section \ref{sec:extremal} treats the extremal distance
Laplacian spectral radius at fixed chromatic number.
Section \ref{section-new} gives additional consequences of the main principle, and Section \ref{conclusion} contains concluding
remarks.

\section{Existing results}\label{section 2}

We first recall the edge-deletion monotonicity of the distance Laplacian spectrum.

\begin{theorem}[Aouchiche--Hansen \cite{AouchicheHansen2013}]\label{thm:monotone}
Let $G$ be connected with $m\ge n$ edges, and let $G^\ast$ be obtained from $G$ by deleting an edge in such a way that
$G^\ast$ remains connected. Then
\[
    \partial^{L}_i(G^\ast)\ge \partial^{L}_i(G)\qquad (1\le i\le n).
\]
\end{theorem}

\noindent Equivalently, adding an edge to a connected graph cannot increase any distance Laplacian eigenvalue. Indeed, if $H=G+e$, then
$G$ is obtained from $H$ by deleting $e$ and remaining connected, so Theorem \ref{thm:monotone} gives
$\partial_i^L(G)\ge \partial_i^L(H)$ for every $i$.

\medskip
\noindent The next result gives the distance Laplacian spectrum of complete multipartite graphs.

\begin{lemma}[\cite{BilalsurveyDL}]\label{lem:multipartite-spectrum}
Let $K_{\ell_1,\dots,\ell_k}$ be a complete $k$-partite graph on $n=\sum_{j=1}^k \ell_j$ vertices, with
$\ell_1\ge\ell_2\ge\cdots\ge\ell_k\ge 1$, and let $p=|\{j:\ell_j\ge 2\}|.$ 
Then the distance Laplacian spectrum is
\[
    \Big( (n+\ell_1)^{(\ell_1-1)},(n+\ell_2)^{(\ell_2-1)},\dots,(n+\ell_p)^{(\ell_p-1)},n^{(k-1)},0\Big),
\]
where exponents denote multiplicities.
\end{lemma}

\noindent Theorem \ref{thm:monotone} shows that distance Laplacian eigenvalues can only move upward under edge deletions. Thus, when a graph is completed by adding edges, the distance Laplacian spectrum moves downward. Lemma \ref{lem:multipartite-spectrum}
provides the sharp comparison object for the chromatic refinements in our study.

\medskip
\noindent  We shall also use the following quadratic form identity. It is standard but useful enough to state explicitly.
\begin{lemma}\label{lem:quadratic-form}
For every connected graph $G$ and every vector $x\in\mathbb{R}^{V(G)}$,
\[
    x^{T}\DL(G)x=\sum_{\{u,v\}\subseteq V(G)} d_G(u,v)(x_u-x_v)^2.
\]
In particular, $\DL(G)$ is positive semidefinite and $\DL(G)\mathbf{1}=\mathbf{0}$.
\end{lemma}

\medskip
\noindent  We also recall two local-structure principles related to cliques and independent sets with identical external neighborhoods.

\begin{theorem}[Clique twins \cite{AouchicheHansen2013}]\label{thm:clique-twins}
Let $G$ be a graph on $n$ vertices, and let $H=\{v_1,\dots,v_s\}$ be a clique such that
\[
    N_G(v_i)\setminus H=N_G(v_j)\setminus H\qquad\text{for all }i,j.
\]
Then $\Tr_G(v_1)=\cdots=\Tr_G(v_s)$, and $\Tr_G(v_1)+1$ is an eigenvalue of $\DL(G)$ with multiplicity at least $s-1$.
\end{theorem}

\begin{theorem}[Independent twins \cite{AouchicheHansen2013}]\label{thm:indep-twins}
Let $G$ be a graph on $n$ vertices, and let $M=\{v_1,\dots,v_s\}$ be an independent set such that
\[
    N_G(v_i)=N_G(v_j)\qquad\text{for all }i,j.
\]
Then $\Tr_G(v_1)=\cdots=\Tr_G(v_s)$, and $\Tr_G(v_1)+2$ is an eigenvalue of $\DL(G)$ with multiplicity at least $s-1$.
\end{theorem}

\section{A color-class majorization principle}\label{section 3}

Throughout, $\mu_G(I)$ denotes the number of eigenvalues of $\DL(G)$ that lie in the real interval $I$, counted with algebraic
multiplicity. Since the eigenvalues are ordered as
$\partial^{L}_1(G)\ge\cdots\ge\partial^{L}_n(G)=0$, we have
\begin{equation}\label{eq:count-identity}
    \mu_G([0,b_\chi)) + \mu_G([b_\chi,\partial^{L}_1(G)])=n.
\end{equation}
 Let $\chi=\chi(G)$, and fix an optimal proper $\chi$-coloring with color classes
$C_1,\dots,C_\chi$. Write
\[
    \ell_1=|C_1|\ge \ell_2=|C_2|\ge\cdots\ge\ell_\chi=|C_\chi|\ge 1,
    \qquad \sum_{j=1}^{\chi}\ell_j=n.
\]
Let $K_{\ell_1,
\dots,
\ell_\chi}$ be the complete $\chi$-partite graph induced by this partition. Since each $C_j$ is an
independent set and every edge of $G$ joins two distinct color classes, the complete multipartite graph contains every possible
cross-edge. Hence $E(G)\subseteq E(K_{\ell_1,\dots,\ell_\chi}),$ so $G$ is a spanning subgraph of $K_{\ell_1,\dots,\ell_\chi}$. The sequence $(\ell_1,
\dots,
\ell_\chi)$ depends on the chosen
optimal coloring, and one such coloring is fixed throughout the section.

\begin{theorem}[Color-class majorization]\label{thm:color-majorization}
Let $G$ be connected, let $\chi=\chi(G)$, and let
$\ell_1\ge\cdots\ge\ell_\chi$ be the color-class sizes of a fixed optimal $\chi$-coloring. Then
 $\partial^{L}_i(G)\ge n+\ell_1$ for $1\le i\le \ell_1-1. $
More generally, let $s_j=\sum_{t=1}^{j}(\ell_t-1)$ for $j\ge1,$ and $s_0=0,$ and let $p=|\{j:\ell_j\ge2\}|$. Then
$\partial^{L}_i(G)\ge n+\ell_j$  whenever $s_{j-1}+1\le i\le s_j,$ for  $1\le j\le p.$ 
\end{theorem}

\begin{proof}
	Let $H=K_{\ell_1,\dots,\ell_\chi}$ be the complete $\chi$-partite graph whose parts are the color classes
	$C_1,\dots,C_\chi$ of the fixed optimal coloring of $G$. Since each $C_j$ is an independent set in $G$, no edge of $G$ has both endpoints in the same color class. Hence every edge of $G$ joins two distinct color classes. The graph $H$ contains all possible edges between distinct color classes, and therefore $E(G)\subseteq E(H).$ Thus $G$ is a spanning subgraph of $H$.
	 If $G=H$, then the conclusion follows from Lemma~\ref{lem:multipartite-spectrum}. Suppose now that $G\ne H$. Choose an ordering $e_1,e_2,\dots,e_t$ of the edges in $E(H)\setminus E(G)$, and define $G_0=G,$ $G_r=G+\{e_1,\dots,e_r\}$ for $1\le r\le t.$ Then $G=G_0\subseteq G_1\subseteq\cdots\subseteq G_t=H,$ and $G_{r+1}$ is obtained from $G_r$ by adding the edge $e_{r+1}$. Since $G$ is connected and $G$ is a spanning subgraph of every $G_r$, each $G_r$ is connected. Moreover, for each $r$, the graph $G_r$ is obtained from $G_{r+1}$ by deleting the single edge $e_{r+1}$, and this deletion leaves the graph connected. Therefore, the edge-deletion monotonicity idea of Theorem~\ref{thm:monotone} can be applied to the pair
	 $(G_{r+1},G_r).$ It follows that 
	\[
	\partial_i^L(G_r)\ge \partial_i^L(G_{r+1})
	\qquad
	1\le i\le n,\;0\le r\le t-1.
	\]
	Chaining these inequalities, we have
	\begin{equation}\label{eq:domination-G-H}
		\partial_i^L(G)=\partial_i^L(G_0)
		\ge \partial_i^L(G_1)
		\ge \cdots
		\ge \partial_i^L(G_t)
		=
		\partial_i^L(H)
		\qquad (1\le i\le n).
	\end{equation}
	Thus every distance Laplacian eigenvalue of $G$, in the ordered list, is at least the corresponding eigenvalue of the complete multipartite comparison graph $H$. By Lemma~\ref{lem:multipartite-spectrum}, the explicit spectrum of $H$ is
	\[
	\mathrm{Spec}(\DL(H))
	=
	\Big(
	(n+\ell_1)^{(\ell_1-1)},
	(n+\ell_2)^{(\ell_2-1)},
	\dots,
	(n+\ell_p)^{(\ell_p-1)},
	n^{(\chi-1)},
	0
	\Big),
	\]
	where $p=|\{j:\ell_j\ge2\}|.$ 
	Here a term with $\ell_j=1$ contributes no eigenvalue of the form $n+\ell_j$, since its multiplicity would be $\ell_j-1=0$.
	As  $\ell_1\ge\ell_2\ge\cdots\ge\ell_\chi,$ the numbers $n+\ell_1,\; n+\ell_2,\;\dots,\; n+\ell_p$ are in nonincreasing order. Also, since $\ell_j\ge2$ for $1\le j\le p$, each of these numbers is strictly larger than $n$. Hence the first part of the ordered spectrum of $\DL(H)$ consists exactly of the blocks
	\[
	\Big((n+\ell_1)^{(\ell_1-1)},
	(n+\ell_2)^{(\ell_2-1)},
	\dots,
	(n+\ell_p)^{(\ell_p-1)}\Big).
	\]
	 In particular, the first $\ell_1-1$ eigenvalues of $\DL(H)$ are equal to $n+\ell_1$, that is,
	\[
	\partial_1^L(H)=\partial_2^L(H)=\cdots=
	\partial_{\ell_1-1}^L(H)=n+\ell_1.
	\]
	Combining this with \eqref{eq:domination-G-H}, we obtain
	\[
	\partial_i^L(G)\ge \partial_i^L(H)=n+\ell_1
	\qquad
	(1\le i\le \ell_1-1).
	\]
	This proves the first assertion. For the general block statement, let  $s_j=\sum_{t=1}^{j}(\ell_t-1),$ with  $s_0=0.$ 
	The first block of the ordered spectrum of $\DL(H)$ has length $\ell_1-1$, the second block has length $\ell_2-1$, and so on. Therefore, for every $1\le j\le p$, the eigenvalue $n+\ell_j$ occupies precisely the positions
	 $s_{j-1}+1,\;s_{j-1}+2,\;\dots,\;s_j$ in the ordered spectrum of $\DL(H)$. Equivalently,
	 $\partial_i^L(H)=n+\ell_j$  whenever $s_{j-1}+1\le i\le s_j.$ Applying \eqref{eq:domination-G-H}, we have 
	\[
	\partial_i^L(G)\ge \partial_i^L(H)=n+\ell_j
	\qquad
	\text{whenever }s_{j-1}+1\le i\le s_j,
	\quad 1\le j\le p.
	\]
\end{proof}

\medskip
\noindent Define the chromatic threshold $b_\chi=n+\ceil{\tfrac{n}{\chi}}.$  The following consequence shows that many distance Laplacian eigenvalues exceed this threshold.

\begin{corollary}\label{cor:many-above-bchi}
For every connected graph $G$ with $\chi(G)=\chi\le n-1$,
\[
    \left|\left\{i\in\{1,
\dots,n\}:\partial_i^L(G)\ge b_\chi\right\}\right|
    \ge \ell_1-1\ge \ceil{\frac{n}{\chi}}-1.
\]
Equivalently,
\[
    \mu_G([0,b_\chi))\le n-\left(\ceil{\frac{n}{\chi}}-1\right)
    =n-\ceil{\frac{n}{\chi}}+1.
\]
\end{corollary}

\begin{proof}
	Let $V(G)=C_1\dot\cup C_2\dot\cup\cdots\dot\cup C_\chi$ be the fixed optimal $\chi$-coloring, with
	\[
	\ell_1=|C_1|\ge \ell_2=|C_2|\ge\cdots\ge \ell_\chi=|C_\chi|.
	\]
	By Theorem~\ref{thm:color-majorization}, the first $\ell_1-1$ distance Laplacian eigenvalues satisfy
	 $\partial_i^L(G)\ge n+\ell_1$ for  $1\le i\le \ell_1-1.$ Since the color classes form a partition of the $n$ vertices into $\chi$ nonempty sets, their average size is $n/\chi$. Therefore the largest color class has size at least the ceiling of this average, that is,
	\begin{equation}\label{eq:ell1-pigeonhole}
		\ell_1\ge \ceil{\frac{n}{\chi}}.
	\end{equation}
	It follows that, for every $1\le i\le \ell_1-1$, we have
	\[
	\partial_i^L(G)
	\ge n+\ell_1
	\ge n+\ceil{\frac{n}{\chi}}
	=
	b_\chi.
	\]
	Thus, the eigenvalues $\partial_1^L(G),\partial_2^L(G),\dots,\partial_{\ell_1-1}^L(G)$ 	all lie in the interval
	 $[b_\chi,\partial_1^L(G)].$ 	Hence, we have  $\left|\left\{i\in\{1,\dots,n\}:\partial_i^L(G)\ge b_\chi\right\}\right|
	\ge \ell_1-1.$ By \eqref{eq:ell1-pigeonhole}, we also have
	\[
	\ell_1-1\ge \ceil{\frac{n}{\chi}}-1.
	\]
	 It remains only to translate this into the equivalent upper bound for the number of eigenvalues below $b_\chi$. Since the eigenvalues are counted with algebraic multiplicity and are ordered from largest to smallest, every eigenvalue lies either in
	 $[0,b_\chi)$ or in $[b_\chi,\partial_1^L(G)].$ 
	Therefore, we obtain
	\[
	\mu_G([0,b_\chi))
	+
	\mu_G([b_\chi,\partial_1^L(G)])
	=
	n.
	\]
	The inequality just proved gives
	\[
	\mu_G([b_\chi,\partial_1^L(G)])
	\ge
	\ell_1-1
	\ge
	\ceil{\frac{n}{\chi}}-1.
	\]
	Consequently, we obtain
	\[
	\mu_G([0,b_\chi))
	=
	n-\mu_G([b_\chi,\partial_1^L(G)])
	\le
	n-\left(\ceil{\frac{n}{\chi}}-1\right).
	\]
	Thus, we have  $\mu_G([0,b_\chi))
	\le
	n-\ceil{\tfrac{n}{\chi}}+1$ 
\end{proof}

\medskip
\noindent The next corollary gives a direct higher-index refinement. The classical second-eigenvalue assertion of
Pirzada and Khan \cite{PirzadaKhan2021} is recalled in the final sentence to avoid a gap in the index range.

\begin{corollary}\label{cor:generalize-2-5}
Let $G$ be connected with $n\ge4$ and chromatic number $\chi\le n-1$. Then, for every integer $k$ satisfying
 $2\le k\le \ceil{\tfrac{n}{\chi}}-1,$ we have $ \partial_k^L(G)\ge b_\chi.$ In particular, when $\ceil{n/\chi}\ge3$ the assertion includes $k=2$.
\end{corollary}

\begin{proof}
Let $k$ satisfy $2\le k\le \ceil{n/\chi}-1$. By \eqref{eq:ell1-pigeonhole},
$\ell_1\ge\ceil{n/\chi}$, and hence
\[
    k\le \ceil{\frac{n}{\chi}}-1\le \ell_1-1.
\]
Therefore, from Theorem \ref{thm:color-majorization}, we have
\[
    \partial_k^L(G)\ge n+\ell_1\ge n+\ceil{\frac{n}{\chi}}=b_\chi.
\]
The final statement records the endpoint case supplied by \cite{PirzadaKhan2021}.
\end{proof}

\medskip
\noindent We note that,  in the remaining range $\chi\le n-2$ with
$\ceil{n/\chi}=2$, the second-eigenvalue bound $\partial_2^L(G)\ge b_\chi$ in Corollary \ref{cor:generalize-2-5}  is precisely the theorem of Pirzada and Khan \cite{PirzadaKhan2021}.

\medskip 
\noindent  The following consequence refines the distribution viewpoint in Theorem 3.4 of \cite{PirzadaKhan2021}.

\begin{corollary}\label{cor:generalize-3-4}
Let $G$ be connected with chromatic number $\chi\le n-1$, and let $\ell_1$ be the largest color-class size in a fixed optimal
$\chi$-coloring. Then
\[
    \mu_G([b_\chi,\partial_1^L(G)])\ge \ell_1-1\ge \ceil{\frac{n}{\chi}}-1.
\]
Moreover, if $\ell_1\ge4$, then $\mu_G([b_\chi,\partial_1^L(G)])\ge3.$ 
\end{corollary}

\begin{proof}
By Corollary \ref{cor:many-above-bchi}, at least $\ell_1-1$ eigenvalues satisfy $\partial_i^L(G)\ge b_\chi$. Since
$\partial_1^L(G)$ is the largest eigenvalue, all such eigenvalues lie in $[b_\chi,\partial_1^L(G)]$. The inequality
$\ell_1-1\ge\ceil{n/\chi}-1$ follows from \eqref{eq:ell1-pigeonhole}. If $\ell_1\ge4$, then $\ell_1-1\ge3$.
\end{proof}

\medskip
\noindent We next recall a basic fact connecting the eigenvalue $n$ with the complement graph.

\begin{theorem}[Aouchiche--Hansen \cite{AouchicheHansen2013}]\label{thm:n-mult-complement}
Let $G$ be connected on $n$ vertices, and let $c(\comp{G})$ be the number of connected components of the complement graph.
Then $n$ is a distance Laplacian eigenvalue of $G$ with multiplicity
 $m_G(n)=c(\comp{G})-1.$ 
\end{theorem}

\noindent The following corollary sharpens the degree-$n-1$ distribution bound by identifying the controlling invariant as
$c(\comp{G})$.
\begin{corollary}\label{cor:upper-by-complement}
Let $G$ be connected with chromatic number $\chi\le n-1$. Then
\[
    \mu_G([b_\chi,\partial_1^L(G)])\le n-c(\comp{G}).
\]
In particular, if $G$ has $P$ universal vertices, that is, vertices of degree $n-1$, then $c(\comp{G})\ge P+1$, and hence
\[
    \mu_G([b_\chi,\partial_1^L(G)])\le n-P-1.
\]
\end{corollary}

\begin{proof}
Since $\chi\le n-1$, we have $\ceil{n/\chi}\ge2$, so $b_\chi>n$. Thus all eigenvalues equal to $n$ lie strictly below
$b_\chi$. By Theorem \ref{thm:n-mult-complement}, there are $c(\comp{G})-1$ such eigenvalues. In addition,
$\partial_n^L(G)=0<b_\chi$. Hence at least $c(\comp{G})$ eigenvalues lie in $[0,b_\chi)$, and at most
$n-c(\comp{G})$ eigenvalues can lie in $[b_\chi,\partial_1^L(G)]$. If $G$ has $P$ universal vertices, then those vertices are
isolated in $\comp{G}$. Since $G$ is not complete when $\chi\le n-1$, the complement also contains at least one further component,
and therefore $c(\comp{G})\ge P+1$.
\end{proof}

\section{Structural results via twin cliques and twin independent sets}\label{section 4}

We now revisit clique and independent-set distribution results by comparing the forced eigenvalues
$\Tr_G(v)+1$ and $\Tr_G(v)+2$ with the chromatic threshold $b_\chi$, using neighborhood-compression parameters.

\noindent If $H$ is a clique and $N_G(v)\setminus H=N$ for every $v\in H$, then every vertex of
$R=V(G)\setminus(H\cup N)$ is nonadjacent to every vertex of $H$; otherwise it would belong to $N$. Consequently,
$d_G(v,u)\ge2$ for all $v\in H$ and $u\in R$.

\begin{theorem}\label{thm:clique-refine}
Let $G$ be connected on $n$ vertices with chromatic number $\chi$. Assume that $G$ contains a clique $H$ of size $s\ge2$ such
that $N_G(v)\setminus H=N$ for every $v\in H, $ for some fixed set $N\subseteq V(G)\setminus H$. Let
$R=V(G)\setminus(H\cup N)$, so $|R|=n-s-|N|$. Then the following hold.
\begin{enumerate}[label=\textup{(\alph*)},noitemsep]
    \item $\DL(G)$ has an eigenvalue $\lambda_H=\Tr_G(v)+1$ with multiplicity at least $s-1$, for any $v\in H$.
    \item The lower estimate $\lambda_H\ge 2n-s-|N|$ holds.
    \item If $s+|N|\le n-\ceil{n/\chi}$, then $\lambda_H\ge b_\chi$, and hence
     $ \mu_G([b_\chi,\partial_1^L(G)])\ge s-1.$ 
\end{enumerate}
\end{theorem}

\begin{proof}
(a) The hypothesis says that the vertices of $H$ form a clique of twins in the sense of Theorem \ref{thm:clique-twins}.
Therefore $\Tr_G(v)$ is constant on $H$, and $\Tr_G(v)+1$ is a distance Laplacian eigenvalue with multiplicity at least $s-1$.
We denote this eigenvalue by $\lambda_H$. (b) For fixed $v\in H$, the vertices in $H\setminus\{v\}$ and in $N$ are adjacent to $v$, while every vertex in $R$ is nonadjacent to
$v$ and hence has distance at least $2$ from $v$. Thus, we have
\[
\Tr_G(v)
\ge (s-1)+|N|+2|R|.
\]
Since $|R|=n-s-|N|$, so we obtain
\[
\Tr_G(v)\ge (s-1)+|N|+2(n-s-|N|)=2n-s-|N|-1.
\]
Therefore, we obtain
\[
\lambda_H=\Tr_G(v)+1\ge 2n-s-|N|.
\]
 (c) If $s+|N|\le n-\ceil{n/\chi}$, then
\[
2n-s-|N|
\ge 2n-\left(n-\ceil{\frac{n}{\chi}}\right)
=n+\ceil{\frac{n}{\chi}}=b_\chi.
\]
By (b), $\lambda_H\ge b_\chi$. Since $\lambda_H$ has multiplicity at least $s-1$ and is no larger than
$\partial_1^L(G)$, at least $s-1$ eigenvalues lie in $[b_\chi,\partial_1^L(G)]$.
\end{proof}

\medskip
\noindent  If $M$ is an independent set and $N_G(v)=N$ for every $v\in M$, then for distinct $x,y\in M$ one has $d_G(x,y)=2$. Indeed,
$x$ and $y$ are nonadjacent, while $N\neq\varnothing$, since $G$ is connected and $s\ge2$, so a common neighbor in $N$ gives a
path of length $2$.

\begin{theorem}\label{thm:indep-refine}
Let $G$ be connected on $n$ vertices with chromatic number $\chi$. Assume that $G$ contains an independent set $M$ of size
$s\ge2$ such that $N_G(v)=N$ for every $v\in M,$ for some fixed set $N\subseteq V(G)\setminus M$. Let $R=V(G)\setminus(M\cup N)$. Then the following hold.
\begin{enumerate}[label=\textup{(\alph*)},noitemsep]
    \item $\DL(G)$ has an eigenvalue $\lambda_M=\Tr_G(v)+2$ with multiplicity at least $s-1$, for any $v\in M$.
    \item The lower estimate $\lambda_M\ge 2n-|N|$ holds.
    \item If $|N|\le n-\ceil{n/\chi}$, then $\lambda_M\ge b_\chi$, and hence
    \[
        \mu_G([b_\chi,\partial_1^L(G)])\ge s-1.
    \]
\end{enumerate}
\end{theorem}

\begin{proof}
(a) By assumption, the vertices of $M$ are pairwise nonadjacent and have identical neighborhoods. Theorem
\ref{thm:indep-twins} therefore gives that $\Tr_G(v)$ is constant on $M$, and that $\Tr_G(v)+2$ is a distance Laplacian
eigenvalue with multiplicity at least $s-1$. We denote this eigenvalue by $\lambda_M$.\\ (b) For fixed $v\in M$, and for every $u\in M\setminus\{v\}$, the preceding observation gives $d_G(v,u)=2$. Also, $d_G(v,u)=1$ for every $u\in N$, and $d_G(v,u)\ge2$ for every $u\in R$, since such a vertex is neither $v$ nor adjacent to
$v$. Therefore, we have
\[
\Tr_G(v)
\ge 2(s-1)+|N|+2|R|.
\]
Since $|R|=n-s-|N|$, this gives
\[
\Tr_G(v)\ge 2(s-1)+|N|+2(n-s-|N|)=2n-|N|-2.
\]
Consequently, $\lambda_M=\Tr_G(v)+2\ge 2n-|N|.$\newline (c) If $|N|\le n-\ceil{n/\chi}$, then
\[
2n-|N|
\ge 2n-\left(n-\ceil{\frac{n}{\chi}}\right)
=n+\ceil{\frac{n}{\chi}}=b_\chi.
\]
By (b), $\lambda_M\ge b_\chi$. Since $\lambda_M$ has multiplicity at least $s-1$, these occurrences lie in
$[b_\chi,\partial_1^L(G)]$, and the desired estimate follows.
\end{proof}

\begin{remark}
Theorems \ref{thm:clique-refine} and \ref{thm:indep-refine} replace earlier size restrictions, such as
$|H|\ge n/2$ or $|M|\ge (n-3)/2$ in distribution statements of \cite{PirzadaKhan2021}, by compression quantities. For clique
twins, the decisive parameter is $s+|N|$, and  for independent twins, it is $|N|$. Even small twin classes can force multiple
eigenvalues above $b_\chi$ when their external neighborhoods are sufficiently small.
\end{remark}

\section{Minimum distance Laplacian spectral radius at fixed chromatic number}\label{sec:extremal}

We now consider the minimum possible value of $\partial_1^L(G)$ among connected graphs with fixed order and fixed chromatic
number. Theorem \ref{thm:monotone} implies that adding edges to a connected graph cannot increase any distance Laplacian
eigenvalue. Consequently, minimizers should be sought among edge-maximal graphs with the prescribed chromatic number. The next
statement records the equality case of the chromatic bound in a form suitable for our purposes; it is the equality part of the
sharp bound in \cite{AouchicheHansen2017}, expressed through the multipartite completion.

\noindent  We use the equality condition in the chromatic bound of Aouchiche and Hansen \cite{AouchicheHansen2017}, equality occurs precisely when the graph is the complete multipartite graph obtained from an optimal coloring whose largest class has size $\ceil{n/\chi}$.
\begin{lemma}\label{lem:equality-AH}
Let $G$ be connected on $n$ vertices with $2\le\chi(G)\le n-1$. Then $  \partial_1^L(G)=n+\ceil{\tfrac{n}{\chi(G)}}$ holds if and only if $G$ is a complete $\chi(G)$-partite graph whose largest part has size $\ceil{n/\chi(G)}$. 
\end{lemma}

\begin{proof}
	Let $\chi=\chi(G)$ and let $b_\chi=n+\ceil{\tfrac{n}{\chi}}.$ If  $V(G)=C_1\dot\cup C_2\dot\cup\cdots\dot\cup C_\chi$ is an optimal $\chi$-coloring with  $\ell_j=|C_j|,$ and $\ell_1\ge \ell_2\ge\cdots\ge \ell_\chi.$ 
	Since the $C_j$ form a partition of $V(G)$ into $\chi$ nonempty classes, the pigeonhole principle gives
	\begin{equation}\label{eq:equality-ell1-lower}
		\ell_1\ge \ceil{\frac{n}{\chi}}.
	\end{equation}
	 Let $H=K_{\ell_1,\dots,\ell_\chi}$ be the complete $\chi$-partite graph with parts $C_1,\dots,C_\chi$. Since each color class is independent in $G$, every edge of $G$ joins two different color classes. The graph $H$ contains all possible such cross-edges, and hence $E(G)\subseteq E(H).$ Thus $G$ is a spanning subgraph of $H$. By the edge-addition form of Theorem~\ref{thm:monotone}, adding missing cross-edges cannot increase the distance Laplacian eigenvalues. Therefore
	\begin{equation}\label{eq:equality-reduction}
		\partial_1^L(G)\ge \partial_1^L(H).
	\end{equation}
	On the other hand, Lemma~\ref{lem:multipartite-spectrum} gives
	 $\partial_1^L(H)=n+\ell_1.$ Combining this with \eqref{eq:equality-ell1-lower} and \eqref{eq:equality-reduction}, we obtain
	\[
	\partial_1^L(G)
	\ge
	\partial_1^L(H)
	=
	n+\ell_1
	\ge
	n+\ceil{\frac{n}{\chi}}.
	\]
	 Now suppose that equality holds. Then every inequality in the chain above must be an equality. In particular,
	 $n+\ell_1=n+\ceil{\frac{n}{\chi}},$ 
	and hence
	\begin{equation}\label{eq:equality-largest-class}
		\ell_1=\ceil{\frac{n}{\chi}}.
	\end{equation}
	Moreover, we note that $\partial_1^L(G)=\partial_1^L(H).$  Next, we show that this forces $G=H$. Suppose, to the contrary, that $G\ne H$. Then some cross-edge between two different color classes is missing from $G$. Equivalently, there exist vertices
	 $x\in C_a, $ and $y\in C_b,$ for $ a\ne b,$  such that $xy\notin E(G)$, while $xy\in E(H)$.
	 Adding this edge gives a connected graph $G'=G+xy$ on the same vertex set and with the same color partition. In the standard proof of the chromatic bound, equality in the comparison $\partial_1^L(G)\ge \partial_1^L(H)$ can occur only when no such cross-edge is missing. Indeed, adding a missing cross-edge strictly shortens at least the distance between its two endpoints and may shorten further distances as well. In the Rayleigh-quotient form, we have
	\[
	z^T\DL(G)z
	=
	\sum_{\{u,v\}\subseteq V(G)} d_G(u,v)(z_u-z_v)^2,
	\]
	this strict decrease in at least one distance is incompatible with equality in the extremal comparison unless every possible cross-edge is already present. Thus equality in $\partial_1^L(G)\ge \partial_1^L(H)$ forces
	$E(G)=E(H),$ and consequently, $G\cong K_{\ell_1,\dots,\ell_\chi}.$ Together with \eqref{eq:equality-largest-class}, this proves that equality in the chromatic bound implies that $G$ is a complete $\chi$-partite graph whose largest part has size
 	$\ceil{\frac{n}{\chi}}.$ 
	
	Conversely, suppose that $G\cong K_{\ell_1,\dots,\ell_\chi}$ and that $\ell_1=\ceil{\frac{n}{\chi}}.$ Then Lemma~\ref{lem:multipartite-spectrum} gives
 \begin{equation*}
		\partial_1^L(G)=n+\ell_1
	=
	n+\ceil{\frac{n}{\chi}}.
	\end{equation*}
	Thus such graphs attain equality. 	It remains only to justify the equivalent balanced formulation. Since
	\[
	\ell_1=\max_{1\le j\le\chi}\ell_j=\ceil{\frac{n}{\chi}},
	\]
	so no part can have size larger than $\ceil{n/\chi}$. If some part had size strictly smaller than
	 $\floor{\tfrac{n}{\chi}},$ 
	then, because the total sum of the $\chi$ part sizes is $n$, another part would have to exceed
	$\ceil{n/\chi}$, contradicting the maximality condition above. Hence every part has size either
	 $\floor{\tfrac{n}{\chi}}$ or $\ceil{\tfrac{n}{\chi}}.$ Therefore the parts are as balanced as possible. This completes the proof.
\end{proof}

\medskip
\noindent  The next theorem identifies the exact minimum of the distance Laplacian spectral radius among graphs with fixed order and chromatic number, and characterizes all extremal graphs.
\begin{theorem}\label{thm:min-spectral-radius}
Among all connected graphs $G$ on $n$ vertices with $\chi(G)=\chi$, where $2\le\chi\le n-1$, the minimum possible distance
Laplacian spectral radius is
\[
    \min_{\substack{G~\mathrm{connected}\\ |V(G)|=n,\,\chi(G)=\chi}} \partial_1^L(G)
    =n+\ceil{\frac{n}{\chi}}.
\]
Moreover, equality holds if and only if $G$ is a complete $\chi$-partite graph
$K_{\ell_1,\dots,\ell_\chi}$ whose largest part has size
\[
    \max_{1\le j\le \chi}\ell_j=\ceil{\frac{n}{\chi}}.
\]
Equivalently, the part sizes are balanced: each part has size $\floor{n/\chi}$ or $\ceil{n/\chi}$.
\end{theorem}

\begin{proof}
	Let $G$ be any connected graph of order $n$ with $\chi(G)=\chi$, where
	$2\le \chi\le n-1$. By the chromatic lower bound \eqref{eq:AH-bound}, we have
	\[
	\partial_1^L(G)\ge n+\ceil{\frac{n}{\chi}}.
	\]
	Thus $n+\ceil{\tfrac{n}{\chi}}$ is a lower bound for the distance Laplacian spectral radius over the whole class of connected graphs on $n$ vertices with chromatic number $\chi$. It remains to show that this lower bound is attainable. Let  $n=q\chi+r,$ for 0$\le r<\chi.$  Consider the balanced complete $\chi$-partite graph $T_\chi(n)=K_{\ell_1,\dots,\ell_\chi},$ where $r$ parts have size $q+1$ and the remaining $\chi-r$ parts have size $q$. This graph is connected because $\chi\ge2$, and its chromatic number is exactly $\chi$, since a complete $\chi$-partite graph with all parts nonempty requires one color for each part.
	 By Lemma~\ref{lem:multipartite-spectrum}, the largest distance Laplacian eigenvalue of a complete multipartite graph
	$K_{\ell_1,\dots,\ell_\chi}$ is
	\[
	\partial_1^L(K_{\ell_1,\dots,\ell_\chi})=n+\max_{1\le j\le\chi}\ell_j.
	\]
	For the balanced graph $T_\chi(n)$, the largest part has size $\max_j\ell_j=
	\ceil{\tfrac{n}{\chi}}.$ 
	Indeed, if $r=0$, then all parts have size $q=n/\chi$, while if $1\le r<\chi$, then the largest parts have size $q+1=\ceil{n/\chi}$. Therefore,
	\[
	\partial_1^L(T_\chi(n))
	=
	n+\ceil{\frac{n}{\chi}}.
	\]
	Hence the lower bound is attained, and consequently
	\[
	\min_{\substack{G~\mathrm{connected}\\ |V(G)|=n,\,\chi(G)=\chi}}
	\partial_1^L(G)
	=
	n+\ceil{\frac{n}{\chi}}.
	\]
	 We now characterize the equality cases. Suppose first that a connected graph $G$ with $\chi(G)=\chi$ satisfies
	\[
	\partial_1^L(G)=n+\ceil{\frac{n}{\chi}}.
	\]
	Then equality holds in the chromatic bound \eqref{eq:AH-bound}. By Lemma~\ref{lem:equality-AH}, this is possible only when $G$ is a complete $\chi$-partite graph $G\cong K_{\ell_1,\dots,\ell_\chi}$ whose largest part has size
	\[
	\max_{1\le j\le\chi}\ell_j=\ceil{\frac{n}{\chi}}.
	\]
	 Conversely, suppose that $G\cong K_{\ell_1,\dots,\ell_\chi}$ and that
	\[
	\max_{1\le j\le\chi}\ell_j=\ceil{\frac{n}{\chi}}.
	\]
	Then Lemma~\ref{lem:multipartite-spectrum} gives
	\[
	\partial_1^L(G)
	=
	n+\max_{1\le j\le\chi}\ell_j
	=
	n+\ceil{\frac{n}{\chi}},
	\]
	so such a graph attains the minimum. Finally, the condition on the largest part is equivalent to the usual balanced condition. Since the $\chi$ part sizes are positive integers whose sum is $n$, no part can exceed
	 $\ceil{\tfrac{n}{\chi}}.$ If some part had size smaller than $\floor{\tfrac{n}{\chi}},$ then, in order for the total sum of the $\chi$ part sizes to remain equal to $n$, another part would have to be larger than
	 $\ceil{\tfrac{n}{\chi}},$ which is impossible. Thus every part has size either
	 $\floor{\frac{n}{\chi}}$ or $\ceil{\frac{n}{\chi}}.$ Therefore the extremal graphs are precisely the balanced complete $\chi$-partite graphs.
\end{proof}

\medskip
\noindent The following consequence recovers Theorem 4.2 of \cite{PirzadaKhan2021}.

\begin{corollary}\label{cor:recover-4-2}
If $n/2\le \chi\le n-1$, then $\ceil{n/\chi}=2$, and Theorem \ref{thm:min-spectral-radius} gives
\[
    \partial_1^L(G)\ge n+2.
\]
Equality holds if and only if
\[
    G\cong K_{\underbrace{2,
\dots,2}_{n-\chi\text{ parts}},\underbrace{1,
\dots,1}_{2\chi-n\text{ parts}}}.
\]
In particular, the minimizer is unique up to isomorphism in this range.
\end{corollary}

\begin{proof}
If $n/2\le\chi\le n-1$, then $1<n/\chi\le2$, so $\ceil{n/\chi}=2$. Theorem
\ref{thm:min-spectral-radius} gives the lower bound $\partial_1^L(G)\ge n+2$, with equality precisely for complete
$\chi$-partite graphs whose largest part has size $2$. Let $a$ be the number of parts of size $2$. Since all parts have size
$1$ or $2$,
\[
    n=2a+1(\chi-a)=a+\chi,
\]
so $a=n-\chi$. Thus there are $n-\chi$ parts of size $2$ and $2\chi-n$ parts of size $1$. The multiset of part sizes is fixed,
and complete multipartite graphs are determined up to isomorphism by this multiset.
\end{proof}

\medskip
\noindent  Aouchiche and Hansen \cite{AouchicheHansen2017} proved that, for $G\ne K_n$, the number of distance Laplacian eigenvalues below
$b_\chi$ satisfies $\mu_G([0,b_\chi))\le n-1$, and that if $\diam(G)\ge3$, then
$\mu_G([0,b_\chi))\le n-2$. Combining these results with Corollary \ref{cor:many-above-bchi} gives a unified refinement.

\begin{theorem}\label{thm:diameter-refine}
Let $G$ be connected on $n\ge5$ vertices with chromatic number $\chi\le n-1$, and let
$b_\chi=n+\ceil{n/\chi}$. Then the following hold.
\begin{enumerate}[label=\textup{(\alph*)}]
    \item
    \[
        \mu_G([0,b_\chi))\le n-\ceil{\frac{n}{\chi}}+1.
    \]
    Equivalently,
    \[
        \mu_G([b_\chi,\partial_1^L(G)])\ge \ceil{\frac{n}{\chi}}-1.
    \]
    \item If $\diam(G)\ge3$, then
    \[
        \mu_G([0,b_\chi))\le n-\max\left\{2,\ceil{\frac{n}{\chi}}-1\right\}.
    \]
    Equivalently,
    \[
        \mu_G([b_\chi,\partial_1^L(G)])
        \ge \max\left\{2,\ceil{\frac{n}{\chi}}-1\right\}.
    \]
\end{enumerate}
\end{theorem}

\begin{proof}
	As the distance Laplacian eigenvalues of $G$ are ordered, and every eigenvalue of $\DL(G)$ lies either below $b_\chi$ or in the interval $[b_\chi,\partial_1^L(G)]$. Hence
	\begin{equation}\label{eq:diam-proof-count}
		\mu_G([0,b_\chi))+
		\mu_G([b_\chi,\partial_1^L(G)])=n.
	\end{equation}
	\noindent\textup{(a)}
	By Corollary~\ref{cor:many-above-bchi}, at least $\ceil{\frac{n}{\chi}}-1$ distance Laplacian eigenvalues of $G$ are not smaller than $b_\chi$. Equivalently,
	\[
	\mu_G([b_\chi,\partial_1^L(G)])
	\ge
	\ceil{\frac{n}{\chi}}-1.
	\]
	Substituting this inequality into \eqref{eq:diam-proof-count}, we get
	\[
	\mu_G([0,b_\chi))
	=
	n-\mu_G([b_\chi,\partial_1^L(G)])
	\le
	n-\left(\ceil{\frac{n}{\chi}}-1\right).
	\]
	Therefore
	\[
	\mu_G([0,b_\chi))
	\le
	n-\ceil{\frac{n}{\chi}}+1.
	\]
	This proves the first assertion, and the lower bound for
	$\mu_G([b_\chi,\partial_1^L(G)])$ is precisely its equivalent form under the counting identity
	\eqref{eq:diam-proof-count}.
	
	\noindent \textup{(b)}
	Now assume that $\diam(G)\ge3.$ By Theorem~2.7 of Aouchiche and Hansen~\cite{AouchicheHansen2017}, the diameter condition forces
	\[
	\mu_G([0,b_\chi))\le n-2.
	\]
	Using \eqref{eq:diam-proof-count}, this is equivalent to
	\[
	\mu_G([b_\chi,\partial_1^L(G)])\ge2.
	\]
	On the other hand, part \textup{(a)}, or equivalently Corollary~\ref{cor:many-above-bchi}, gives the independent color-class estimate
	\[
	\mu_G([b_\chi,\partial_1^L(G)])
	\ge
	\ceil{\frac{n}{\chi}}-1.
	\]
	Thus the number of eigenvalues lying in
	$[b_\chi,\partial_1^L(G)]$ is bounded below by both quantities
	 $2$ and $\ceil{\tfrac{n}{\chi}}-1.$ 
	Consequently, we have
	\[
	\mu_G([b_\chi,\partial_1^L(G)])
	\ge
	\max\left\{2,\ceil{\frac{n}{\chi}}-1\right\}.
	\]
	Applying the counting identity \eqref{eq:diam-proof-count}, we have
	\[
	\mu_G([0,b_\chi))
	=
	n-\mu_G([b_\chi,\partial_1^L(G)])
	\le
	n-\max\left\{2,\ceil{\frac{n}{\chi}}-1\right\}.
	\]
	This proves the second assertion.
\end{proof}

\medskip
\noindent  Part (a) of Theorem \ref{thm:diameter-refine} refines Theorem 2.6 of \cite{AouchicheHansen2017}, and part (b) refines Theorem
2.7 of \cite{AouchicheHansen2017}. The improvement is most visible when $\chi$ is small, because then
$\ceil{n/\chi}-1$ is large.

\medskip

\noindent  We provide representative computations, rounded to three decimal places where necessary, verifying sharpness of the bounds and
the strengthened multiplicity claims. For complete multipartite graphs, the spectra follow from Lemma
\ref{lem:multipartite-spectrum}. For the remaining graphs, the eigenvalues were computed from
$\DL(G)=\operatorname{diag}(\Tr_G(v_1),\dots,\Tr_G(v_n))-D(G)$ and then sorted in nonincreasing order. In each row,
$m_{\ge b_\chi}$ denotes the number of eigenvalues $\partial_i^L(G)$ satisfying $\partial_i^L(G)\ge b_\chi$. Here
$S_{6,2}$ denotes the double star on $8$ vertices whose two centers have degrees $6$ and $2$; its complement
$\comp{S_{6,2}}$ has diameter $3$ and chromatic number $6$. For $K_{4,4,2}$ and $K_{2,2,1,1,1}$, the exact spectra follow
from Lemma \ref{lem:multipartite-spectrum}.

\begin{table}[H]
    \centering
    \small
    \begin{tabular}{@{}lcccccc@{}}
        \toprule
        Graph $G$ & $n$ & $\chi$ & $\diam(G)$ & $b_\chi$ & $m_{\ge b_\chi}$ & First six $\partial_i^L(G)$ \\
        \midrule
        $K_{2,2,1,1,1}$ & 7 & 5 & 2 & 9 & 2 &
        $9,\,9,\,7,\,7,\,7,\,7$ \\
        $\comp{S_{6,2}}$ & 8 & 6 & 3 & 10 & 2 &
        $16.429,\,10.922,\,9,\,9,\,9,\,9$ \\
        $P_8$ & 8 & 2 & 7 & 12 & 7 &
        $38.446,\,28,\,25.016,\,22,\,19.787,\,18$ \\
        $K_{4,4,2}$ & 10 & 3 & 2 & 14 & 6 &
        $14,\,14,\,14,\,14,\,14,\,14$ \\
        \bottomrule
    \end{tabular}
    \caption{Numerical verification of Theorems \ref{thm:color-majorization} and \ref{thm:diameter-refine}.}
    \label{table 1}
\end{table}

\noindent  The examples in Table \ref{table 1} may be read as follows.
\begin{enumerate}[noitemsep]
    \item $K_{2,2,1,1,1}$ shows that the basic lower bound can be attained sharply at the threshold. Here
    $\ceil{7/5}-1=1$, while two parts of size $2$ give two eigenvalues equal to $b_\chi=9$.
    \item $\comp{S_{6,2}}$ illustrates the diameter phenomenon in Theorem \ref{thm:diameter-refine}(b): although
    $\ceil{8/6}-1=1$, the condition $\diam(G)\ge3$ forces at least two eigenvalues to be at least $b_\chi$.
    \item $P_8$ shows that for small $\chi$ the majorization bound is strong. Here $\ceil{8/2}-1=3$, but in fact seven
    eigenvalues exceed $b_\chi=12$.
    \item $K_{4,4,2}$ exhibits the block structure predicted by Theorem \ref{thm:color-majorization}: the largest part size is
    $4$, so at least $3$ eigenvalues are $n+4=14$, and in fact there are $6$ such eigenvalues because two parts have size $4$.
\end{enumerate}

\begin{figure}[H]
    \centering
    \begin{tikzpicture}[scale=1, every node/.style={circle, draw, inner sep=1.5pt}]
        \node (a1) at (0,1) {};
        \node (a2) at (0,0) {};
        \node (b1) at (1.2,1.3) {};
        \node (b2) at (1.2,-0.3) {};
        \node (c1) at (2.4,1) {};
        \node (d1) at (2.7,0.5) {};
        \node (e1) at (2.4,0) {};
        \foreach \u in {a1,a2}{
            \foreach \v in {b1,b2,c1,d1,e1}{
                \draw (\u) -- (\v);
            }
        }
        \foreach \u in {b1,b2}{
            \foreach \v in {c1,d1,e1}{
                \draw (\u) -- (\v);
            }
        }
        \foreach \u/\v in {c1/d1,c1/e1,d1/e1}{
            \draw (\u) -- (\v);
        }
        \node[draw=none, rectangle, inner sep=0pt] at (1.2,-0.8) {$K_{2,2,1,1,1}$};
    \end{tikzpicture}
    \hspace{1cm}
    \begin{tikzpicture}[scale=1, every node/.style={circle, draw, inner sep=1.5pt}]
        \node (u) at (0,0) {};
        \node (v) at (1.6,0) {};
        \draw (u)--(v);
        \foreach \i/\y in {1/1.2,2/0.6,3/0,4/-0.6,5/-1.2}{
            \node (lu\i) at (-1.2,\y) {};
            \draw (u)--(lu\i);
        }
        \node (lv) at (2.8,0) {};
        \draw (v)--(lv);
        \node[draw=none, rectangle, inner sep=0pt] at (0.8,-1.8) {$S_{6,2}$};
    \end{tikzpicture}
    \hspace{1cm}
    \begin{tikzpicture}[scale=1, every node/.style={circle, draw, inner sep=1.5pt}]
        \node (x) at (0,0) {};
        \node (y) at (1.4,0) {};
        \draw (x)--(y);
        \foreach \i/\pos in {1/1.2,2/0.4,3/-0.4,4/-1.2}{
            \node (u\i) at (0.7,\pos) {};
            \draw (u\i)--(x);
            \draw (u\i)--(y);
        }
        \node (z) at (-1.0,0) {};
        \draw (z)--(x);
        \node[draw=none, rectangle, inner sep=0pt] at (0.2,-1.8) {$G_{\mathrm{ind}}$};
    \end{tikzpicture}
    \hspace{1cm}
    \begin{tikzpicture}[scale=1, every node/.style={circle, draw, inner sep=1.5pt}]
        \node (a) at (0,0.8) {};
        \node (b) at (0,-0.3) {};
        \node (c) at (0.8,1.3) {};
        \draw (a)--(b)--(c)--(a);
        \node (p) at (1.6,0.8) {};
        \node (q) at (1.6,-0.3) {};
        \foreach \u in {a,b,c}{
            \draw (\u)--(p);
            \draw (\u)--(q);
        }
        \draw (p)--(q);
        \node (t1) at (2.7,0.8) {};
        \node (t2) at (2.7,-0.3) {};
        \node (t3) at (1.9,-0.3) {};
        \draw (p)--(t1)--(t2)--(t3);
        \node[draw=none, rectangle, inner sep=0pt] at (2.0,-1.0) {$G_{\mathrm{clq}}$};
    \end{tikzpicture}
    \caption{Graphs used in Tables \ref{table 1} and \ref{table 2}.}
    \label{fig 1}
\end{figure}
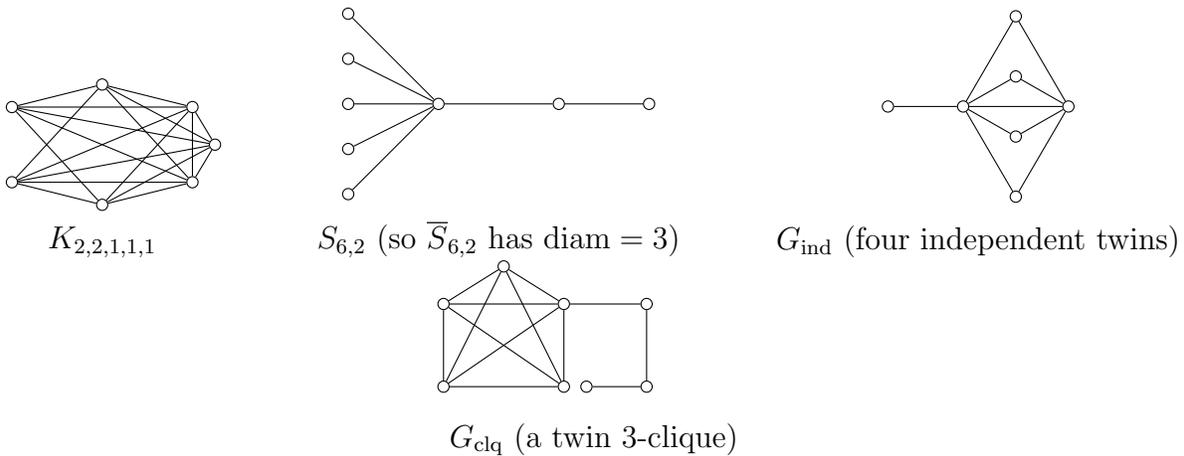

\noindent  The next table gives examples tailored to the extremal theorem, the diameter refinement, and the twin-structure mechanisms.

\begin{table}[H]
    \centering
    \small
    \begin{tabular}{@{}lcccccc@{}}
        \toprule
        Graph $G$ & $n$ & $\chi$ & $\diam(G)$ & $b_\chi$ & $m_{\ge b_\chi}$ & First six $\partial_i^L(G)$ \\
        \midrule
        $K_{3,3,3,2}$ & 11 & 4 & 2 & 14 & 6 &
        $14,\,14,\,14,\,14,\,14,\,14$ \\
        $C_{10}$ & 10 & 2 & 5 & 15 & 9 &
        $35.472,\,35.472,\,26.528,\,26.528,\,26,\,25$ \\
        $K_{3,5}$ & 8 & 2 & 2 & 12 & 4 &
        $13,\,13,\,13,\,13,\,11,\,11$ \\
        $G_{\mathrm{ind}}$ & 7 & 3 & 2 & 10 & 4 &
        $13,\,12,\,12,\,12,\,8,\,7$ \\
        $G_{\mathrm{clq}}$ & 8 & 5 & 4 & 10 & 7 &
        $26.610,\,18.512,\,14,\,14,\,14,\,13.808$ \\
        \bottomrule
    \end{tabular}
    \caption{Further examples supporting the refined theorems.}
    \label{table 2}
\end{table}

\noindent In Table \ref{table 2}, the graph $K_{3,3,3,2}$ attains $\partial_1^L=b_\chi$, illustrating Theorem
\ref{thm:min-spectral-radius}. The cycle $C_{10}$ illustrates Theorem \ref{thm:diameter-refine}(b) for graphs with diameter at
least $3$. The graphs $G_{\mathrm{ind}}$ and $G_{\mathrm{clq}}$ are twin-structure examples:
$G_{\mathrm{ind}}$ contains four independent twins with common neighborhood of size $2$ and exhibits the forced eigenvalue
$\Tr_G(v)+2=12$ with multiplicity $3$; $G_{\mathrm{clq}}$ contains a $3$-clique of twins with
$\Tr_G(v)+1=14$ appearing with multiplicity at least $2$.

\noindent The concrete descriptions of the twin examples are as follows.
\begin{enumerate}[noitemsep]
    \item $G_{\mathrm{ind}}$ is obtained by taking an edge on vertices $\{x,y\}$ and adding four vertices
    $u_1,\dots,u_4$, each adjacent to both $x$ and $y$. Thus $N(u_i)=\{x,y\}$ for all $i$. Finally, add one vertex $z$
    adjacent to $x$ only. Then $u_1,\dots,u_4$ are independent twins and $12$ appears in the distance Laplacian spectrum with
    multiplicity $3$.
    \item $G_{\mathrm{clq}}$ is obtained from a triangle $H=\{a,b,c\}$ by adding two vertices $p,q$ adjacent to all vertices of
    $H$ and to each other, and then attaching a path of length $3$ from $p$. The vertices $a,b,c$ form a twin clique with
    $\Tr_G(a)=\Tr_G(b)=\Tr_G(c)=13$, and hence $\Tr_G(a)+1=14$ occurs with multiplicity at least $2$.
\end{enumerate}

\section{Additional consequences of color-class domination}\label{section-new}

In this section, we collect some further consequences of the color-class domination principle. These consequences are particularly useful when the complete color-class profile of an optimal coloring is known, rather than only the largest color class.

\medskip
\noindent The next theorem counts how many large color classes force eigenvalues above an arbitrary threshold $n+t$.
\begin{theorem}\label{thm:multi-threshold}
	Let $G$ be connected with chromatic number $\chi$, and let
	$\ell_1\ge\ell_2\ge\cdots\ge\ell_\chi$ be the color-class sizes in a fixed optimal $\chi$-coloring of $G$. For an integer $t\ge2$, let
	$A_t=\sum_{\substack{1\le j\le\chi\\ \ell_j\ge t}}(\ell_j-1).$ 
	Then $\mu_G([n+t,\partial^{L}_1(G)])\ge A_t$ or equivalently, whenever $A_t>0$, then
	 $\partial^{L}_{A_t}(G)\ge n+t.$ 
\end{theorem}

\begin{proof}
	Let $q=\max\{j:\ell_j\ge t\},$ provided that the set is nonempty. Since the sequence
	$\ell_1,\dots,\ell_\chi$ is nonincreasing, the color classes of size at least
	$t$ are exactly $C_1,C_2,\dots,C_q.$ Thus $A_t=\sum_{j=1}^{q}(\ell_j-1).$ Now, by Theorem~\ref{thm:color-majorization}, if
 	$s_j=\sum_{h=1}^{j}(\ell_h-1),$ then for every $1\le j\le p$, where $p$ is the number of color classes of size at least $2$, we have
	 $\partial_i^{L}(G)\ge n+\ell_j$  whenever $ s_{j-1}+1\le i\le s_j.$ Since $t\ge2$, every class counted in $A_t$ is among the first $p$ classes. Therefore, for all indices $1\le i\le A_t=s_q,$ the corresponding lower bound is at least
	 $n+\ell_q\ge n+t.$ 
	Hence, we have  $\partial_i^{L}(G)\ge n+t$ for  $1\le i\le A_t.$ It follows that at least $A_t$ distance Laplacian eigenvalues lie in the interval $[n+t,\partial^{L}_1(G)]$, that is,
	\[
	\mu_G([n+t,\partial^{L}_1(G)])\ge A_t.
	\]
\end{proof}

\noindent Taking $t=\ceil{n/\chi}$ gives a refinement of Corollary~\ref{cor:many-above-bchi} whenever several color classes have size at least the average ceiling.

\begin{corollary} \label{cor:refined-threshold-count}
	With the notation of Theorem~\ref{thm:multi-threshold},
	\[
	\mu_G([b_\chi,\partial^{L}_1(G)])
	\ge
	\sum_{\substack{1\le j\le\chi\\ \ell_j\ge \ceil{n/\chi}}}(\ell_j-1)
	\ge
	\ell_1-1
	\ge
	\ceil{\frac{n}{\chi}}-1.
	\]
\end{corollary}

\begin{proof}
	If  $t=\ceil{\frac{n}{\chi}},$ then  $b_\chi=n+\ceil{\frac{n}{\chi}}=n+t.$ The first inequality follows directly from Theorem~\ref{thm:multi-threshold}. Since the largest color class satisfies
	 $\ell_1\ge\ceil{\frac{n}{\chi}},$ so the summation on the right contains at least the term $\ell_1-1$. Hence
	\[
	\sum_{\substack{1\le j\le\chi\\ \ell_j\ge \ceil{n/\chi}}}(\ell_j-1)
	\ge \ell_1-1.
	\]
	Finally, the pigeonhole principle gives $\ell_1-1\ge \ceil{\frac{n}{\chi}}-1.$ Combining these inequalities proves the result.
\end{proof}

\noindent For balanced complete multipartite graphs, the preceding count can be computed exactly.
\begin{corollary}\label{cor:balanced-count}
	Let $T_\chi(n)$ denote the balanced complete $\chi$-partite graph on $n$ vertices, and let
	 $n=q\chi+r,$ for  $0\le r<\chi.$ Then
	\[
	\mu_{T_\chi(n)}([b_\chi,\partial^{L}_1(T_\chi(n))])=
	\begin{cases}
		n-\chi, & r=0,\\[1mm]
		rq, & 1\le r<\chi.
	\end{cases}
	\]
\end{corollary}

\begin{proof}
	In the balanced complete $\chi$-partite graph $T_\chi(n)$, the part sizes are as equal as possible. First suppose $r=0$. Then every part has size $q$, and
	\[
	b_\chi=n+\ceil{\frac{n}{\chi}}=n+q.
	\]
	By Lemma~\ref{lem:multipartite-spectrum}, each part of size $q$ contributes $q-1$ eigenvalues equal to $n+q$. Therefore, we have
	\[
	\mu_{T_\chi(n)}([b_\chi,\partial^{L}_1(T_\chi(n))])
	=
	\chi(q-1)
	=
	q\chi-\chi
	=
	n-\chi.
	\]
	When $q=1$, this formula gives $0$, which is consistent with the fact that $T_n(n)=K_n$ has no distance Laplacian eigenvalue equal to $n+1$. Now suppose $1\le r<\chi$. Then $T_\chi(n)$ has $r$ parts of size $q+1$ and $\chi-r$ parts of size $q$. In this case $\ceil{\frac{n}{\chi}}=q+1,$ so $b_\chi=n+q+1.$ By Lemma~\ref{lem:multipartite-spectrum}, only the parts of size $q+1$ contribute eigenvalues at the level $n+q+1$. Each such part contributes
	 $(q+1)-1=q$ eigenvalues. Since there are $r$ such parts, the total number of eigenvalues in
	$[b_\chi,\partial^{L}_1(T_\chi(n))]$ is $rq.$  
\end{proof}

\medskip
\noindent The next statement is a Ky Fan type form of the same domination principle. It is sometimes more informative than a single eigenvalue count, since it gives lower bounds for partial sums at the top of the distance Laplacian spectrum.

\begin{theorem}\label{thm:ky-fan}
	Let $G$, $\chi$, and $\ell_1\ge\ell_2\ge\cdots\ge\ell_\chi$ 
	be as in Theorem~\ref{thm:color-majorization}. Form the ordered list
	 $\Theta=(\theta_1,\dots,\theta_{n-\chi})$ by writing $n+\ell_j$ exactly $\ell_j-1$ times for every $j$ with $\ell_j\ge2$, and then arranging the resulting numbers in nonincreasing order. Then, for every
	$1\le r\le n-\chi$, 
	\[
	\sum_{i=1}^{r}\partial_i^{L}(G)\ge \sum_{i=1}^{r}\theta_i.
	\]
	In particular,
	\[
	\sum_{i=1}^{n-\chi}\partial_i^{L}(G)
	\ge
	\sum_{j=1}^{\chi}(\ell_j-1)(n+\ell_j),
	\]
	where a term with $\ell_j=1$ contributes zero.
\end{theorem}

\begin{proof}
	Since $\sum_{j=1}^{\chi}(\ell_j-1)=n-\chi,$ so the list $\Theta$ has exactly $n-\chi$ entries. By Theorem~\ref{thm:color-majorization}, the eigenvalues of $\DL(G)$ dominate, component by component, the corresponding block values coming from the complete multipartite comparison graph
	$K_{\ell_1,\dots,\ell_\chi}$. More explicitly, for each $j$ with $\ell_j\ge2$, the block of indices associated with $C_j$ satisfies
	$\partial_i^{L}(G)\ge n+\ell_j.$ After arranging these block values in nonincreasing order, we have 
	 $\partial_i^{L}(G)\ge\theta_i$ for every $1\le i\le n-\chi.$ Now, summing the first $r$ inequalities, we have
	\[
	\sum_{i=1}^{r}\partial_i^{L}(G)
	\ge
	\sum_{i=1}^{r}\theta_i,
	\]
	for every $1\le r\le n-\chi$. With $r=n-\chi$, we derive
	\[
	\sum_{i=1}^{n-\chi}\partial_i^{L}(G)
	\ge
	\sum_{i=1}^{n-\chi}\theta_i.
	\]
	Finally, by the definition of $\Theta$, we have
	\[
	\sum_{i=1}^{n-\chi}\theta_i
	=
	\sum_{j=1}^{\chi}(\ell_j-1)(n+\ell_j),
	\]
	with the convention that the contribution is zero when $\ell_j=1$.
\end{proof}

\medskip
\noindent The complement-component obstruction from Corollary~\ref{cor:upper-by-complement} also applies at every threshold strictly larger than $n$. Combining it with Theorem~\ref{thm:multi-threshold} gives the following two-sided estimate.

\begin{corollary}\label{cor:sandwich}
	Let $G$ be connected, and let $t\ge1$ be an integer. Then
	\[
	\mu_G([n+t,\partial^{L}_1(G)])\le n-c(\comp{G}).
	\]
	Consequently, for every integer $t\ge2$,
	\[
	A_t
	\le
	\mu_G([n+t,\partial^{L}_1(G)])
	\le
	n-c(\comp{G}),
	\]
	where $A_t$ is defined in Theorem~\ref{thm:multi-threshold}.
\end{corollary}

\begin{proof}
	Since $t\ge1$, we have  $n<n+t.$ Therefore every distance Laplacian eigenvalue equal to $n$ lies below the interval
	$[n+t,\partial^{L}_1(G)]$. By Theorem~\ref{thm:n-mult-complement}, the eigenvalue $n$ has multiplicity
	 $c(\comp{G})-1.$ In addition, $\partial_n^{L}(G)=0$, and clearly
	 $0<n+t.$ 	Thus at least $(c(\comp{G})-1)+1=c(\comp{G})$ eigenvalues of $\DL(G)$ lie below $n+t$. Hence at most
	 $n-c(\comp{G})$ eigenvalues can lie in the interval $[n+t,\partial^{L}_1(G)]$. This proves that 
	\[
	\mu_G([n+t,\partial^{L}_1(G)])\le n-c(\comp{G}).
	\]
	 For $t\ge2$, Theorem~\ref{thm:multi-threshold} gives the lower bound
	\[
	\mu_G([n+t,\partial^{L}_1(G)])\ge A_t.
	\]
\end{proof}

\begin{remark}
	Theorem~\ref{thm:multi-threshold} and Corollary~\ref{cor:sandwich} show that the upper part of the distance Laplacian spectrum is controlled by two competing structural effects. Large color classes push eigenvalues upward, while disconnectedness of the complement forces copies of the eigenvalue $n$, thereby limiting the number of eigenvalues that can lie above any threshold larger than $n$.
\end{remark}

\section{Concluding remarks and further directions}\label{conclusion}

Theorems \ref{thm:color-majorization} and \ref{thm:min-spectral-radius} show that optimal colorings encode substantial spectral
information for $\DL(G)$. A largest color class forces a block of large distance Laplacian eigenvalues, and the extremal
spectral-radius problem at fixed chromatic number is attained by balanced complete multipartite graphs. On the distribution
side, Theorem \ref{thm:diameter-refine} combines chromatic and diameter effects, while Theorems \ref{thm:multi-threshold} and
\ref{thm:ky-fan} show that the entire color-class profile gives a more detailed picture than the single value $\ell_1$.

\noindent  Several natural questions remain. One may try to sharpen diameter-sensitive lower bounds on the number of eigenvalues exceeding
$b_\chi$ for large diameter, characterize equality in Corollary \ref{cor:many-above-bchi} beyond complete multipartite families,
and extend the color-class majorization method to normalized distance Laplacian variants. Another promising direction is to
study stability, that is, if $\partial_1^L(G)$ is close to $n+\ceil{n/\chi}$, must $G$ be close, in an edit-distance sense, to a balanced
complete multipartite graph?

\section*{Declarations}
\noindent\textbf{Data Availability:} There is no data associated with this article.

\noindent\textbf{Funding:} The author did not receive support from any organization for the submitted work.

\noindent\textbf{Conflict of interest:} The author has no competing interests to declare that are relevant to the content of this article.

\end{document}